\documentclass[12pt]{article}
\usepackage{a4}
\usepackage{latexsym}
\usepackage{amsfonts}
\usepackage{amssymb}
\usepackage{amsmath}
\usepackage{amsthm}

\newtheorem{theorem}{Theorem}
\newtheorem{proposition}[theorem]{Proposition}
\newtheorem{definition}[theorem]{Definition}
\newtheorem{lemma}[theorem]{Lemma}
\newtheorem{corollary}[theorem]{Corollary}

\author{Timothy G. F. Jones}
\date{}
\title{Further improvements to incidence and Beck-type bounds over prime finite fields}

\begin{document}

\maketitle

\begin{abstract}
We establish improved finite field Szemer\'edi-Trotter and Beck type theorems. First we show that if $P$ and $L$ are a set of points and lines respectively in the plane $\mathbb{F}_p^2$, with $|P|,|L| \leq N$ and $N<p$, then there are at most $C_1 N^{\frac{3}{2}-\frac{1}{662}+o(1)}$ incidences between points in $P$ and lines in $L$. Here $C_1$ is some absolute constant greater than $1$. This improves on the previously best-known bound of $C_1 N^{\frac{3}{2}-\frac{1}{806}+o(1)}$.

Second we show that if $P$ is a set of points in $\mathbb{F}_p^2$ with $|P|<p$ then either at least $C_2|P|^{1-o(1)}$ points in $P$ are contained in a single line, or $P$ determines least $C_2 |P|^{1+\frac{1}{109}-o(1)}$ distinct lines. Here $C_2$ is an absolute constant less than $1$. This improves on previous results in two ways. Quantitatively, the exponent of $1+\frac{1}{109}-o(1)$ is stronger than the previously best-known exponent of $1+\frac{1}{267}$. And qualitatively, the result applies to all subsets of $\mathbb{F}_p^2$ satisfying the cardinality condition; the previously best-known result applies only when $P$ is of the form $P=A \times A$ for $A \subseteq \mathbb{F}_p$. 
\end{abstract}

\section{Introduction}

This paper proves new results concerning two types of incidence problem between points and lines in the plane determined by a finite field $\mathbb{F}_p$ of prime order $p$. 

Throughout, we use  $Y=O(X)$, $X=\Omega(Y)$ and $Y \ll X$ all to mean that there is an absolute constant $C$ with $Y\leq CX$. We use $Y \approx X$ or $Y = \Theta(X)$ to mean $Y \ll X$ and $X \ll Y$. If the implicit constant $C$ is dependent on a parameter $\epsilon$ then we indicate this with a subscript, e.g. $Y \ll_{\epsilon} X$. We also use $Y \lesssim X$, $Y=\widetilde{O}(Y)$ etc. to mean that there is an absolute constant $c$ with $Y\ll \log(X)^{c} X$. 

\subsection{Incidence bounds over a prime finite field}
The first problem considered is that of obtainining efficient bounds on the number $I(P,L)$ of incidences between a set $P$ of points and a set $L$ of lines in the plane $\mathbb{F}_p^2$. If $|P|,|L| \leq N$ then we know by Cauchy-Schwarz that $I(P,L)\ll N^{3/2}$, so non-trivial incidence bounds are of the form $I(P,L)\ll N^{3/2-\epsilon}$ for $\epsilon >0$.

When working over the plane $\mathbb{R}^2$, Szemer\'edi and Trotter \cite{ST} obtained the sharp bound $\epsilon \geq 1/6$. This was extended to $\mathbb{C}$ by Toth \cite{toth} and a near-sharp generalisation to higher dimensional points and varieties was recently given by Solymosi and Tao \cite{solymositao}, subsequently made sharp in the $\mathbb{R}^4$ case by Zhal \cite{zahl}. 

In the finite field case considered here, one must impose nondegeneracy conditions in order to prove nontrivial bounds, since in the case $P= \mathbb{F}_p^2$ it possible to obtain  $I(P,L)\approx N^{3/2}$. When $N<p^{2-\delta}$, Bourgain, Katz and Tao \cite{BKT} proved the existence of an $\epsilon >0$, dependent only on $\delta>0$. Helfgott and Rudnev \cite{HR} then obtained $\epsilon \geq 1/10,678$ in the `small $N$' case $\delta \geq 1$. This was subsequently improved by the present author to $\epsilon \geq 1/806-o(1)$. In the `large $N$' case $\delta<1$, Vinh \cite{vinh} has obtained an explicit lower bound on $\epsilon$ in terms of $\delta$.

In this paper we further improve the state of the art in the `small $N$' case to $\epsilon \geq 1/662-o(1)$:

\begin{theorem}\label{theorem:incidences}
If $P$ and $L$ are a set of points and lines over $\mathbb{F}_p$ with $|P|,|L| \leq N <p$ then $I(P,L)\lesssim N^{\frac{3}{2}-\frac{1}{662}}$.
\end{theorem}

\subsection{Beck-type theorems over a prime finite field}
The second problem considered is that of obtaining Beck-type results over $\mathbb{F}_p$. Given a set $P$ of points in the plane, write $L(P)$ for the set of lines determined by pairs of points. Beck \cite{beck} proved that for any finite $P \subseteq \mathbb{R}^2$ at least one of two things happens. Either $\Omega(|P|)$ of the points are colinear, or $|L(P)|\gg |P|^2$.

Helfgott and Rudnev \cite{HR} proved a form of Beck's theorem for finite fields. In particular case of the direct product $P= A \times A$ they showed that $|L(P)| \gg |P|^{1+ 1/267}$ so long as $|P|<p$. As with the problem of counting incidences, a nondegeneracy condition of this kind is required for results of this type because $|L(P)|\approx p^2$ when $P = \mathbb{F}_p^2$. In fact, Iosevich, Rudnev and Zhai \cite{IRZ} recently showed that $|L(P)|\approx p^2$ whenever $|P|>p \log p$. 

We improve on the Helfgott-Rudnev bound in two respects. First, we improve the exponent $1/267$ to $1/109-o(1)$. Second, we establish the result for general $P \subseteq \mathbb{F}_p^2$ rather than simply those of the form $P=A \times A$: 

\begin{theorem}\label{theorem:beck2}
If $P \subseteq \mathbb{F}_p^2$ and $|P|<p$ then at least one of the following must occur:
\begin{enumerate}
\item At least $\widetilde{\Omega}\left(|P|\right)$ points from $P$ are contained in a single line.
\item $|L(P)|\gtrsim|P|^{1+\frac{1}{109}}$. \end{enumerate}
\end{theorem}

\subsection{Structure of the paper}

The proofs of Theorems \ref{theorem:incidences} and \ref{theorem:beck2} develop the method in \cite{TJincidences2011}, and   readers are referred to that paper for a sketch of the approach. 

The structure of this paper is as follows. Section \ref{section:refine} shows how to refine a setup of points and lines to a particular configuration. Section \ref{section:interpretation} interprets this configuration as a partial sum-product problem, for which Section \ref{section:bounding} obtains efficient bounds. Finally, Section \ref{section:proofs} uses the results from the previous three sections to prove Theorems \ref{theorem:incidences} and \ref{theorem:beck2}.

We remark that the results of Sections \ref{section:refine} and \ref{section:interpretation} are general and apply when working over any field. Sections \ref{section:bounding} and \ref{section:proofs} are specific to the $\mathbb{F}_p$ setting. 

\section{Refining points and lines}\label{section:refine}

This section shows that if there exist many points that are each incident to many lines then a large set of points must lie in a certain kind of configuration. In Section \ref{section:interpretation} we will interpret this configuration as corresponding to a partial sum-product problem. We begin with a definition:

\begin{definition}
Let $p$ be a point in the plane, and $P$ be a set of points in the plane. We say that the pair $(P,p)$ is \textbf{$K$-good} if $P$ is supported over at most $K$ lines through $p$.
\end{definition}

The main results of this section are the following two propositions.

\begin{proposition}\label{theorem:refine1} Let $P$ and $L$ be a set of points and lines respectively over a plane such that every point in $P$ is incident to $\Theta(K)$ lines in $L$. Suppose that $|P|\gg 1$ and $|P|K^2\gg |L|$. 

Then there exist distinct points $p_1,p_2 \in P$, and a point-set $Q \subseteq P$ with $|Q|\approx \frac{K^4|P|}{|L|^2} $ such that $(Q,p_1)$ and $(Q,p_2)$ are both $O(K)$-good, and the $O(K)$ lines supporting $Q$ are in each case elements of $L$.
\end{proposition}

\begin{proposition}\label{theorem:refine2} Let $Q,p_1,p_2$ be as provided by Proposition \ref{theorem:refine1}. Suppose in addition that each of the $O(K)$ supporting lines through each of $p_1,p_2$ is incident to at most $\frac{|P|}{K}$ points in $Q$ and that $|Q|\gg |L|/K,|P|/K, K$. 

Then there exist distinct points $p_3,p_4 \in Q$ such that $p_2,p_3,p_4$ are colinear along a line $l\in L$ that is not incident to $p_1$, and there exists a point-set $R \subseteq Q$ with $|R| \approx \frac{|P|K^8}{|L|^4}$ such that $(R,p_3)$ and $(R,p_4)$ are both $O(K)$-good.
\end{proposition}

The rest of Section \ref{section:refine} is concerned with proving these propositions. Section \ref{section:refinelemmata} establishes some preliminary incidence lemmata. Section \ref{section:refine1proof} then gives the proof of Proposition \ref{theorem:refine1}, and Section \ref{section:refine2proof} gives the proof of Proposition \ref{theorem:refine2}.

\subsection{Lemmata}\label{section:refinelemmata}

We first record two standard results.

\begin{lemma}\label{theorem:pointrefine}
Let $P_1$ be the set of points in $P$ incident to at least $\frac{I(P,L)}{2|P|}$ lines in $L$. Then $I(P_1,L)\approx I(P,L)$. Similarly, if $L_1$ is the set of lines in $L$ incident to at least $\frac{I(P,L)}{2|L|}$ points in $P$ then $I(P,L_1)\approx I(P,L)$.
\end{lemma}

\begin{proof}
We prove the result for points, leaving that for lines as an exercise. Let $P_2$ be the set of points in $P$ incident to at most $\frac{I(P,L)}{2|P|}$ lines in $L$. Then 

\begin{align*}I(P_2,L)&=\sum_{p \in P_2}\#\left\{l \in L \text{ incident to } p\right\}\leq |P_2| \frac{I(P,L)}{2|P|}\leq \frac{I(P,L)}{2}.
\end{align*} 

Since $I(P,L)=I(P_1,L)+I(P_2,L)$ we obtain $I(P_1,L)\geq \frac{I(P,L)}{2}$ as required.
\end{proof}

\begin{lemma}\label{theorem:pointrefine2}
Let $P_1$ be the set of points in $P$ incident to no more than $\max\left\{4,\frac{4|L|^2}{I(P,L)}\right\}$ lines in $L$. Then $I(P_1,L)\approx I(P,L)$. Similarly, if $L_1$ is the set of lines in $L$ incident to no more than $\max\left\{4,\frac{4|P|^2}{I(P,L)}\right\}$ points in $P$, then $I(P,L_1)\approx I(P,L)$.
\end{lemma}

\begin{proof}
We prove the result for points, leaving that for lines as an exercise. Let $\lambda=\max\left\{4,\frac{4|L|^2}{I(P,L)}\right\}$ and let $P_2$ be the set of points incident to at least $\lambda$ lines in $L$. Then we have

\begin{align*}
I(P_2,L)&=\sum_{p \in P_2}\sum_{l \in L}\delta_{pl}\\
&\leq \frac{1}{\lambda}\sum_{p \in P_2}\sum_{l_1,l_2 \in L}\delta_{pl}\\
&\leq \frac{1}{\lambda}\left(I(P,L)+|L|^2\right)\\
& \leq \frac{I(P,L)}{2}.
\end{align*}

Since $I(P,L)=I(P_1,L)+I(P_2,L)$ we obtain $I(P_1,L)\geq \frac{I(P,L)}{2}$ as required.
\end{proof}

For points $p,q$ in the plane, let $l_{pq}$ be the line determined by $p$ and $q$. We use Lemma \ref{theorem:pointrefine} to establish the following useful result.

\begin{lemma}\label{theorem:cover1}
Let $P$ be a set of points and $L$ a set of lines, such that every point is incident to $\Theta(K)$ lines in $L$. For each $p \in P$, define $P_p=\left\{q \in P:l_{pq}\in L\right\}$. Then there exists $P_1 \subseteq P$ with $|P_1|\approx |P|$ such that $|P_p|\gg \frac{K^2|P|}{|L|}$ for each $p \in P_1$.
\end{lemma}

\begin{proof}
Since every point in $P$ is incident to $\Theta(K)$ lines in $L$ we may write 

$$I(P,L)\approx K|P|.$$

Let $L_1$ be the set of lines in $L$ incident to $\Omega\left(\frac{I(P,L)}{|L|}\right)=\Omega\left(\frac{K|P|}{|L|}\right)$ points in $P$. By Lemma \ref{theorem:pointrefine} we have $I(P,L_1)\approx I(P,L)\approx K|P|$. Now let $P_1$ be the set of points in $P$ incident to $\Omega\left(\frac{I(P,L_1)}{|P|}\right)=\Omega\left(K\right)$ lines in $L_1$. By Lemma \ref{theorem:pointrefine} we have 

\begin{equation}\label{eq:incidences1}
I(P_1,L_1)\approx I(P,L_1)\approx K|P|.
\end{equation}

Since $P_1 \subseteq P$ and $L_1 \subseteq L$ we know that each point in $P_1$ is incident to at most $O(K)$ lines in $L_1$. So we have 

\begin{equation}\label{eq:incidences2}
I(P_1,L_1)\ll K|P_1|.
\end{equation}

Comparing (\ref{eq:incidences1}) and (\ref{eq:incidences2}) we see that $|P_1|\gg |P|$, and so $|P_1|\approx |P|$ since $P_1$ is a subset of $P$. Now, each $p \in P_1$ is incident to $\Omega\left(K\right)$ lines in $L_1$. And each line in $L_1$ is incident to $\Omega\left(\frac{K|P|}{|L|}\right)$ points in $P$. So we have $|P_p|\gg \frac{K^2|P|}{|L|}$ for each $p \in P_1$.
\end{proof}

\subsection{Proof of Proposition \ref{theorem:refine1}}\label{section:refine1proof}
By Lemma \ref{theorem:cover1} there exists $P_1 \subseteq P$ with $|P_1|\approx |P|$ such that $|P_p|\gg \frac{K^2|P|}{|L|}\gg 1$ for each $p \in P_1$. In particular we can find a single $p_1 \in P$ such that $|P_{p_1}|\gg \frac{K^2|P|}{|L|}$. Applying Lemma \ref{theorem:cover1} again, this time to $P_{p_1}$ and $L$, we can find $p_2 \in P_{p_1}$ such that 

$$|P_{p_1}\cap P_{p_2}|\gg \frac{K^2|P_{p_1}|}{|L|}\gg \frac{K^4|P|}{|L|}.$$

Note that $P_{p_1}$ is the set of points incident to lines that are themselves incident to $p_1$. Since $p_1$ is incident to only $O(K)$ lines in $L$ this means that $\left(P_{p_1},p_1\right)$ is $O(K)$-good. Similarly, $\left(P_{p_2},p_2\right)$ is $O(K)$-good. So $\left(P_{p_1}\cap P_{p_2},p_1\right)$ and $\left(P_{p_1}\cap P_{p_2},p_2\right)$ are both $O(K)$-good. We take $Q$ to be an appropriately-sized subset of $P_{p_1}\cap P_{p_2}$.
\qed

\subsection{Proof of Proposition \ref{theorem:refine2}}\label{section:refine2proof}
Since $Q \subseteq P$ we know that every point in $Q$ is incident to $\Theta(K)$ lines in $L$. So by Lemma \ref{theorem:cover1} there exists $Q_1 \subseteq Q$ with $|Q_1|\approx |Q|$ such that $|Q_p|\gg \frac{K^2|Q|}{|L|}$ for each $p \in Q_1$. Now, $Q_1$ is contained in $Q$ and so $(Q_1,p_2)$ is $O(K)$-good, i.e. $Q_1$ is supported over $O(K)$ lines in $L$ that are incident to $p_2$. Let $J \subseteq L$ be this set of $O(K)$ supporting lines. We have $I(Q_1,J)=|Q_1|$ and $|J|\ll K$. Let $J_1$ be the set of $l \in J$ incident to at least $\Omega\left(\frac{|Q|}{K}\right)$ points in $Q_1$. By Lemma \ref{theorem:pointrefine} we know that 

\begin{equation}\label{eq:incidences11}
I(Q_1,J_1)\approx I(Q_1,J)\approx |Q|.
\end{equation}

But since $Q_1\subseteq P$ and $J_1 \subseteq L$, we know that each line in $J_1$ is incident to at most $O\left(\frac{|P|}{K}\right)$ points in $Q$. So we have

\begin{equation}\label{eq:incidences12}
I(Q_1,J_1)\ll \frac{|J_1||P|}{K}.
\end{equation}

Comparing (\ref{eq:incidences11}) and (\ref{eq:incidences12}) gives $|J_1|\gg \frac{|Q|K}{|P|}$. Since $|Q|\gg |P|/K$ by hypothesis we know, by appropriate choice of constants in the statement of the theorem, that $|J_1| \geq 2$. Since $|J_1|\geq 2$ and all lines in $J_1$ are incident to $p_2$ we know that there is at least one line in $J_1$ that is not incident to $p_1$. Fix this line $l^*$, so that $|Q \cap l^*|\gg \frac{|Q|}{K}$, and in particular $|Q \cap l^*|$ has at least, say, $100$ elements since $|Q|\gg K$. We have

\begin{equation}\label{eq:incidences6}
\frac{K^2|Q|}{|L|}\left|Q \cap l^*\right|\ll \sum_{p \in Q \cap l^*}|Q_p|.
\end{equation}

On the other hand by Cauchy-Schwarz we have 

\begin{align*}
\sum_{p \in Q \cap l^*}|Q_p|& \leq |Q|^{1/2}\left(\sum_{p_3,p_4 \in Q\cap l^*}\left|Q_{p_3}\cap Q_{p_4}\right|\right)^{1/2}\\
&= |Q|^{1/2}\left( \sum_{p \in Q \cap l^*}|Q_p|+ \sum_{p_3\neq p_4 \in Q\cap l^*}\left|Q_{p_3}\cap Q_{p_4}\right|\right)^{1/2}
\end{align*}

If the first summation on the right dominates then we have 

$$\sum_{p \in Q \cap l^*}|Q_p|\ll |Q|.$$

Since $|Q \cap l^*|\gg \frac{|Q|}{K}$ and $|Q_p|\gg \frac{K^2|Q|}{|L|}$ we then obtain $|Q|K \gg |L|$ by comparison with (\ref{eq:incidences6}), contradicting the hypothesis $|Q| \ll |L|/K$. So the second summation on the right dominates and we have instead 

$$\frac{K^4 |Q| \left|Q \cap l_*\right|^2}{|L|^2} \ll \sum_{p_3 \neq p_4 \in Q\cap l^*}|Q_{p_3}\cap Q_{p_4}|$$

so there exist distinct $p_3,p_4 \in Q \cap l^*$ such that 

$$|Q_{p_3}\cap Q_{p_4}|\gg \frac{K^4 |Q|}{|L|^2}\gg \frac{K^8|P|}{|L|^4}.$$

We take $R$ to be an appropriately-sized subset of $Q_{p_3}\cap Q_{p_4}$. Since $p_2,p_3,p_4 \in l^*$ and $p_1 \notin l^*$ we have the required result.
\qed

\section{Interpretation as partial sum-products}\label{section:interpretation}

If $A$ and $B$ are subsets of a field, then we write 

$$A+B=\left\{a+b:a \in A, b \in B\right\}.$$

If $G \subseteq A \times B$ then we write

$$A\stackrel{G}{+}B=\left\{a+b:(a,b)\in G\right\}.$$

We extend these definitions to the operations of multiplication, subtraction and division. This section shows that the configuration of points described in Proposition \ref{theorem:refine2} implies the existence of sets $A,B$ and $G \subseteq A \times B$ such that $|A\overset{G}-B|,|A\overset{G}/B| \ll |A|,|B|$. If there are many points configured in this way then $G$ is large. 

\begin{proposition}\label{theorem:reduction}
Suppose $P$ is a set of points in the plane over a field $F$ and $p_1,p_2,p_3,p_4 \notin P$ are points in the plane such that $(P,p_i)$ is $K_i$-good for each $i$. Suppose moreover that $p_2,p_3$ and $p_4$ are colinear on a line $l$ that is not incident to $p_1$ nor to any of the points in $P$. 

Then there exist sets $A,B \subseteq F$ with $|A|\leq K_3$, $|B|\leq K_4$, and $G \subseteq A \times B$ with $|G|=|P|$ such that $|A \overset{G}-B| \leq K_2$ and $|A \overset{G}/B| \leq K_1$.
\end{proposition}

\begin{proof}
Let $\tau$ be a projective transformation of the plane that sends $p_3$ and $p_4$ to the line at infinity, so that lines through $\tau(p_3)$ are parallel to the vertical axis and lines through $\tau(p_4)$ are parallel to the horizontal axis, and sends $p_1$ to the origin. Define $G=\tau(P)\subseteq F^2$.

The set $G$ is supported over $K_3$ vertical lines $K_4$ horizontal lines. Let $A$ be the set of $x$-intercepts of the vertical lines and $B$ be the set of $y$-interecepts of the hortizontal lines, so that $G \subseteq A \times B$ and $|A|\leq K_3$, $|B|\leq K_4$. 

Furthermore, $G$ is supported over $K_1$ lines through the origin. We note that these are identified by their gradient, and that a point $(a,b)\in F^2$ is incident to the line with gradient $\xi$ if and only if $\frac{a}{b}=\xi$. We therefore have $|A \overset{G}/ B|\leq K_1$. 

Finally, $G$ is supported over $K_2$ lines through $\tau(p_2)$. Since $\tau$ sends $p_3$ and $p_4$ to the line at infinity, and $p_2$ is colinear with these points, we know that $\tau(p_2)$ lies on the line at infinity. Say that all lines incident to $\tau(p_2)$ have gradient $\lambda \in F$. These are identified by the intercept, and a point $(a,b)\in F^2$ is incident to the line with gradient $\rho$ if and only if $a+\lambda b = \rho$. We therefore have $|A \overset{G}- \lambda B|\leq K_2$.  

We now let $B'= \lambda B$ and $G'=\left\{(a,\lambda b):(a,b)\in G\right\}$ to obtain $|G'|=|G|=|P|$, $|A|\leq K_3$, $|B'|=|B|\leq K_4$, $|A \overset{G'}- B'|\leq K_1$ and $|A \overset{G}/ B'|\leq K_2$.
\end{proof}

\section{Bounding partial sum-products}\label{section:bounding}

From Sections \ref{section:refine} and \ref{section:interpretation} we know that if too many points in $P$ are incident to too many lines in $L$ then there exist $A,B \subseteq \mathbb{F}_p$ and a large $G \subseteq A \times B$ such that $|A\overset{G}-B|,|A\overset{G}/B|$ are both small relative to $|A|$ and $|B|$. 

In this section, we will show that this is not possible, in the following sense:

\begin{proposition}\label{theorem:partialsumprod}
If $A,B \subseteq \mathbb{F}_p$ and $G \subseteq A \times B$ with $|G| \leq p |B|$ then we have 

$$|G|^{55}\ll |A|^{36}|B|^{37}|A \overset{G}-B|^{28}|A \overset{G}/ B|^8.$$
\end{proposition}

Finite field sum-product estimates are the driving force behind Proposition \ref{theorem:partialsumprod}. These assert that 
 $\max\left\{|A+A|,|A\cdot A|\right\}$ must always be large relative to the cardinality of $A$. We will make use of the most-recent finite field sum-product estimate, due to Rudnev \cite{rudnev}, who showed that $\max\left\{|A+A|,|A\cdot A|\right\} \gtrsim |A|^{12/11}$. As Rudnev notes, the result makes use of the multiplicative energy $E_{\times}(A)$, which is the number of solutions to $ab=cd$ with $a,b,c,d \in A$ rather than the product set itself, and extends perfectly well to difference sets rather than sumsets. We will adopt the following formulation of Rudnev's result.
 
\begin{lemma}[Rudnev]\label{theorem:rudnev}
If $A \subseteq \mathbb{F}_p$ and $|A|<p^{1/2}$ then $E_{\times}(A)^4 \lesssim |A-A|^7 |A|^4.$
\end{lemma}

Lemma \ref{theorem:rudnev} concerns the \textit{complete} difference set $A-A$, but we have control of only the partial difference set $A\overset{G}-B$. We prove the following lemma, which yields a Balog-Szemer\'edi-Gowers type estimate for sumsets and a more efficient bound for multiplicative energy. 

\begin{lemma}\label{theorem:halfbsg}
If $A,B \subseteq \mathbb{F}_p$ and $G \subseteq A \times B$ then there exists $A' \subseteq A$ with $|A'|\gg \frac{|G|}{|B|}$ such that 
\begin{enumerate}
\item $|A'-A'|\ll \frac{|A\overset{G}-B|^4|A|^4|B|^3}{|G|^5}$
\item $E_{\times}(A') \gg \frac{|G|^6}{|B|^5|A|^2|A\overset{G}/B|^2}$.
\end{enumerate}
\end{lemma}

Implicitly the proof of the bound on $A'-A'$ in Lemma \ref{theorem:halfbsg} is the same as that of the Balog-Szemer\'edi-Gowers type result due to Bourgain and Garaev \cite{BG}, although it is presented below slightly differently. 

Combining Lemma \ref{theorem:rudnev} and Lemma \ref{theorem:halfbsg} gives the statement of Proposition \ref{theorem:partialsumprod}. The remainder of this section is therefore concerned with the proof of Lemma \ref{theorem:halfbsg}.

\subsection{Proof of Lemma \ref{theorem:halfbsg}}

We begin with some preliminary lemmata. Given $G \subseteq A \times B$ we define $N(a)$ for each $a \in A$ as $N(a)=\left\{b \in B: (a,b) \in G\right\}.$ The following result can be found in \cite{TV} and is repeated in \cite{BG}.

\begin{lemma}\label{theorem:pathsoflengthtwo}
For sets $G \subseteq A \times B$ and $\epsilon>0$ there exists $A' \subseteq A$ with $|A'|\gg_{\epsilon}\frac{|G|}{|B|}$ such that for $(1-\epsilon)|A'|^2$ pairs $(a_1,a_2)\in A' \times A'$ we have $$\left|N(a_1)\cap N(a_2)\right|\gg \frac{|G|^2}{|A|^2|B|}.$$
\end{lemma}

\begin{lemma}\label{theorem:pathsapp}
For sets $A,B$ and $G \subseteq A \times B$ there exists $A' \subseteq A$ and $H \subseteq A' \times A'$ with $|A'|\gg_{\epsilon} \frac{|G|}{|B|}$ and $|H|\geq (1-\epsilon) |A'|^2$ such that 

\begin{enumerate}
\item $|A'\overset{H}-A'|\ll_{\epsilon} \frac{|A\overset{G}- B|^2 |A|^2 |B|}{|G|^2}.$
\item $|A'\overset{H}/A'|\ll_{\epsilon} \frac{|A\overset{G}/ B|^2 |A|^2 |B|}{|G|^2}.$
\end{enumerate}
\end{lemma}

\begin{proof}
Apply Lemma \ref{theorem:pathsoflengthtwo}. Let $H$ be the set of $(a_1,a_2)\in A' \times A'$ for which $\left|N(a_1)\cap N(a_2)\right|$ so that $|H|\geq (1-\epsilon)|A'|^2$. We prove the result for $|A'\overset{H}-A'|$, leaving that for $|A'\overset{H}/A'|$ as an exercise. 

For each $x \in A' \overset{H}-A'$ fix $\left(a_1(x),a_2(x)\right) \in H$ such that $a_1(x)-a_2(x)=x$. Now let $Y=\left\{(x,b):x \in A' \overset{H}-A', b \in N(a_1(x))\cap N(a_2(x))\right\}.$ It is clear that $|Y|\gg |A' \overset{H}-A'| \frac{|G|^2}{|A|^2|B|}.$ On the other hand the injection 
\begin{align*}
f:&Y \to (A\overset{G}- B) \times (A\overset{G}- B)\\
f:&(x,b)\mapsto (a_1(x)-b,a_2(x)-b)
\end{align*}
shows that $|Y|\leq |A \overset{G}-B|^2$. Comparing the upper and lower bounds on $|Y|$ gives the result.
\end{proof}

The following result can also be found in \cite{TV}, where it is set as exercise 2.5.4. It is not clear that a proof can usually be found in the literature, and so we give one here.

\begin{lemma}\label{theorem:BKTexercise1}
Let $0< \epsilon < 1/4$ and let $G \subseteq A \times B$ and $H \subseteq B \times C$, such that $|G|\geq (1-\epsilon)|A||B|$ and $|H|\geq (1-\epsilon)|B||C|$. Then there exist $A' \subseteq A$ and $C' \subseteq C$ with $|A'|\geq (1-\sqrt{\epsilon})|A|$ and $|C'|\geq (1-\sqrt{\epsilon})|C|$ such that 

$$|A'-C'|\ll_{\epsilon}\frac{|A\overset{G}-B||B\overset{H}-C|}{|B|}.$$
\end{lemma}

\begin{proof}
Let $A'$ be the set of $a \in A$ with a $G$-degree of at least $(1-\sqrt{\epsilon})|B|$, and $C'$ be the set of $c \in C$ with an $H$-degree of at least $(1-\sqrt{\epsilon})|B|$. Since $|G|\geq (1-\epsilon)|A||B|$ and $|H|\geq (1-\epsilon)|B||C|$ we know that $|A'| \geq (1-\sqrt{\epsilon}) |A|$ and $|C'|\geq (1-\sqrt{\epsilon})|C|$. Now if $a \in A'$ and $c \in C'$ then there are at least $(1-2\sqrt{\epsilon})|B|$ elements $b \in B$ such that $(a,b)\in G$ and $(b,c)\in H$. Denote the set of such elements by $B_{ac}$

For each $x \in A'-C'$ fix a single $a(x) \in A', c(x) \in C'$ such that $a(x)-c(x)=x$. Let $Y=\left\{(x,b):x \in A'-C', b \in B_{a(x)c(x)} \right\}.$ It is clear that $|Y|\gg_{\epsilon} |A'-C'||B|$ On the other hand the injection 
\begin{align*}
f:&Y \to (A\overset{G}- B) \times (B\overset{H}- C)\\
f:&(x,b)\mapsto (a(x)-b,b-c(x))
\end{align*}
shows that $|Y|\leq |A \overset{G}-B||B \overset{H}-C|$. Comparing the upper and lower bounds on $|Y|$ gives the result.
\end{proof}

We record a particular consequence of Lemma \ref{theorem:BKTexercise1}, in the case where $A=B=C$ and $G=H$:

\begin{corollary}\label{theorem:BKTexercise2}
Let $0<\epsilon<1/4$ and let $G \subseteq A \times A$ with $|G| \geq (1-\epsilon)|A|^2$. Then there exists $A' \subseteq A$ with $|A'|\geq(1-2\sqrt{\epsilon})|A|$ such that 

$$|A'-A'|\ll_{\epsilon}\frac{|A \overset{G}-A|^2}{|A|}.$$
\end{corollary}

\begin{proof}
By Lemma \ref{theorem:BKTexercise1} there exist $A_1,A_2 \subseteq A$ with $|A_1|,|A_2|\geq (1-\sqrt{\epsilon})|A|$ such that $|A_1-A_2|\ll_{\epsilon}\frac{|A \overset{G}-A|^2}{|A|}.$ We then let $A= A_1 \cap A_2$.
\end{proof}

We now prove Lemma \ref{theorem:halfbsg}. Let $\epsilon>0$ be sufficiently small. By Lemma \ref{theorem:pathsapp} there exists there exists $A' \subseteq A$ and $H \subseteq A' \times A'$ with $|A'|\gg_{\epsilon} \frac{|G|}{|B|}$ and $|H|\geq (1-\epsilon) |A'|^2$ such that $|A'\overset{H}-A'|\ll_{\epsilon} \frac{|A\overset{G}- B|^2 |A|^2 |B|}{|G|^2}$ and $|A'\overset{H}/A'|\ll_{\epsilon} \frac{|A\overset{G}/ B|^2 |A|^2 |B|}{|G|^2}.$ Apply Corollary \ref{theorem:BKTexercise2} to obtain $A'' \subseteq A'$ with $|A''|\geq(1-\epsilon)|A'|$ such that
 
$$|A''-A''|\ll \frac{|A'\overset{H}-A'|^2}{|A'|}\ll \frac{|A\overset{G}-B|^4|A|^4|B|^3}{|G|^5}.$$

Now let $H' = H \cap (A'' \times A'')$. Since both $H$ and $A'' \times A''$ are of cardinality at least $(1-\epsilon)|A'|^2$ we have $|H'|\gg |A'|^2$. So by Cauchy-Schwarz we have

$$E_{\times}(A'')\geq \frac{|H'|^2}{|A''\overset{H'}/A''|} \gg \frac{|A'|^4}{|A'\overset{H}/A'|} \gg \frac{|G|^6}{|B|^5|A|^2|A\overset{G}-B|^2}$$ which completes the proof. \qed

\section{Proving Theorems \ref{theorem:incidences} and \ref{theorem:beck2}}\label{section:proofs}

\subsection{Proof of Theorem \ref{theorem:incidences}}

Suppose that $I(P,L)\gtrsim N^{3/2-\epsilon}$. We shall show that $\epsilon \geq 1/662$. By Lemma \ref{theorem:pointrefine2} we may assume that every line in $L$ is incident to at most $O\left(N^{1/2+\epsilon}\right)$ points in $P$. By a dyadic pigeonholing we may find a subset $P_1 \subseteq P$ and an integer $K$ with 

\begin{equation}\label{eq:incidences7}
|P_1|K \gtrsim N^{3/2- \epsilon}
\end{equation}

such that every point in $P$ is incident to $\Theta(K)$ lines in $L$. Note moreover that since $|P_1|\leq N$ we have 

\begin{equation}\label{eq:incidences8}
K \gtrsim N^{1/2- \epsilon}
\end{equation}

Applying Proposition \ref{theorem:refine1} and then \ref{theorem:refine2} we know that at least one of the following is true:

\begin{enumerate}
\item $|P_1|K^2 \ll |L|.$
\item $|P_1|K^3 \ll |L|^2$
\item $|P_1|K^5 \ll |L|^3.$ 
\item $K^5 \ll |L|^2.$
\item There exists $R \subseteq P_1$ with $|R|\approx \frac{|P_1|K^8}{|L|^4}$ and points $p_1,p_2,p_3,p_4$ such that $(R,p_i)$ is $O(K)$-good for each $i$. Moreover, the points $p_2,p_3,p_4$ are colinear along a line $l$ in $L$ that is not incident to $p_1$.  
\end{enumerate}

We quickly dispense with the first four cases. If $|P_1|K^2 \ll |L| $ then applying (\ref{eq:incidences7}), (\ref{eq:incidences8}) and the fact that $|L|\leq N$ we get $\epsilon \geq 1/2-o(1)$. By the same arguments, the other three cases yield $\epsilon \geq 1/6-o(1)$, $\epsilon \geq 1/10-o(1)$ and  $\epsilon \geq 1/10-o(1)$ respectively.

We are left with the fifth case. Since the line $l$ that is incident to $p_2,p_3,p_4$ is an element of $L$, we know that it is incident to at most $O\left(N^{1/2+\epsilon}\right)$ points in $R$. We may assume that $\left|R \setminus \left\{l\right\}\right|\approx |R|$ or we are already done.

Apply Proposition \ref{theorem:reduction} to $R \setminus \left\{l\right\}$ to obtain $A, B \subseteq \mathbb{F}_p$ with $|A|,|B| \ll K$ and $G \subseteq A \times B$ with $|G|\gg \frac{|P_1|K^8}{|L|^4}$ such that $|A \overset{G}-B|,|A \overset{G}/B|\ll K$.

By Proposition \ref{theorem:partialsumprod} we have $|G|^{55} \ll K^{109}$ and so $|P_1|^{55}K^{331} \ll |L|^{220}.$ Applying (\ref{eq:incidences7}) and (\ref{eq:incidences8}) as before we get $\epsilon \geq 1/662 - o(1)$.
\qed

\subsection{Proof of Theorem \ref{theorem:beck2}}

For $l \in L(P)$, write $\mu(l)$ for the number of points in $P$ incident to $l$. It is clear that

$$|P|^2 \approx \sum_{l \in L(P)}\mu(l)^2.$$

By a dyadic pigeonholing there exists $L_1 \subseteq L(P)$ and an integer $k$ such that $\mu(l)\approx k$ for all $l \in L_1$ and 

\begin{equation}
\label{eq:incidences9}
|L_1|k^2 \approx |P|^2.
\end{equation}

We have $I(P,L_1)\approx |L_1|k$. So there exists $P_1 \subseteq P$ and an integer $K$ such that every point in $P_1$ is incident to $\Theta(K)$ lines in $L_1$ and $|P_1|K \gtrsim |L_1|k$. We therefore have 

\begin{equation}
\label{eq:incidences10}
K \gtrsim \frac{|L_1|k}{|P_1|}.
\end{equation}

Applying Proposition \ref{theorem:refine1} and then Proposition \ref{theorem:refine2} to $P_1, L_1$ and $K$ we know that at least one of the following is true

\begin{enumerate}
\item $|P_1|K^2 \ll |L_1|.$
\item $|P_1|K^3 \ll |L_1|^2$
\item $|P_1|K^5 \ll |L_1|^3.$ 
\item $K^5 \ll |L_1|^2.$
\item There exists $R \subseteq P_1$ with $|R|\approx \frac{|P_1|K^8}{|L_1|^4}$ and points $p_1,p_2,p_3,p_4$ such that $(R,p_i)$ is $O(K)$-good for each $i$. Moreover, the points $p_2,p_3,p_4$ are colinear along a line $l\in L_1$ that is not incident to $p_1$.  
\end{enumerate}

In the first case, we apply (\ref{eq:incidences10}) to obtain $|P_1|\lesssim 1$. The second, third and fourth cases yield respectively $k \lesssim 1$, $k \lesssim 1$ and $k \gtrsim |P|$. If $k \lesssim 1$ then we have $|L(P)|\gtrsim |P|^2$, and if $k \gtrsim |P|$ then there are $\widetilde{\Omega}(|P|)$ colinear points in $P$. 

So we are left with the fifth case. Since $l \in L_1$ it is incident to at most $O(k)$ points in $P$, and so we may assume $\left|R \setminus \left\{l\right\}\right|\approx |R|$ or we are already done. Apply Proposition \ref{theorem:reduction} to obtain $A, B \subseteq \mathbb{F}_p$ with $|A|,|B| \ll K$ and $G \subseteq A \times B$ with $|G|\gg \frac{|P_1|K^8}{|L_1|^4} $ such that $|A \overset{G}-B|,|A \overset{G}/B|\ll K$.

By Proposition \ref{theorem:partialsumprod} we have $|G|^{55} \ll K^{109}$ and so $|P_1|^{55} K^{331} \ll |L_1|^{220}$. By (\ref{eq:incidences10}) this gives $|L_1|^{111}k^{331}\lesssim |P_1|^{276}$. By (\ref{eq:incidences9}) and the fact that $|P_1|\leq |P|$ we then get $k^{109}\lesssim |P|^{54}$ and so by (\ref{eq:incidences9}) again $|P|^{110} \lesssim |L_1|^{109}$. We conclude therefore that $|L(P)|\geq |L_1|\gtrsim |P|^\frac{110}{109}$.
\qed

\bibliographystyle{plain}
\bibliography{Incidencesbibliography}

\end{document}